\documentclass[11pt,a4paper,twoside]{article}

\usepackage{authblk}  % pour faire correctement "authors and affiliations"

\usepackage{ushort} % pour faire du underline en math proprement (\underline produit un trait trop long !)

\usepackage{amsmath}
\usepackage{amssymb}
\usepackage{amsthm}
\usepackage{mathrsfs}
\usepackage{dsfont}
\usepackage{mathtools}
\usepackage{euscript}
\usepackage{xcolor}
\usepackage{comment}
\usepackage{soul}

\usepackage{hyperref}

\usepackage{color}
\usepackage{graphicx}
\usepackage{bbm}
\usepackage{enumerate}
\usepackage{ifpdf}

\usepackage{relsize} % Usage : \mathlarger{\int}, \mathsmaller{\sum}, etc 

\numberwithin{equation}{section}  

\newtheorem{theorem}{Theorem}[section]
\newtheorem{definition}{Definition}[section]
\newtheorem{remark}{Remark}[section]
\newtheorem{lemma}[theorem]{Lemma}

\newtheorem{corollary}[theorem]{Corollary}

\newcommand{\e}{\operatorname{e}} % pour faire des exponentielles avec un "e" droit mais surtour pour que LateX le gere comme un vrai operateur au niveau des espaces

\newcommand{\Lip}{\mathrm{Lip}}

\newcommand{\var}{\mathrm{var}}

\newcommand{\dd}{{\mathrm d}}

\newcommand{\N}{\mathds{N}}

\newcommand{\R}{\mathds{R}}

\newcommand{\E}{\mathds{E}}

\newcommand{\proba}{\mathds{P}}

\newcommand{\un}{\mathds{1}}

\def\Oun{\mathcal{O}(1)}

\def\A{\mathcal{A}}

\newcommand{\lawto}{\mathop{}\!\xrightarrow[]{\mathrm{\scriptscriptstyle{law}}}}

\font\gfont=cmmi10 scaled \magstep{1.5}    
{2}

\newcommand{\gdelta}{\hbox{\gfont \char14}}

\newcommand{\cte}[1]{C_{\mathrm{(\ref{#1})}}}
\newcommand{\pcte}[1]{c_{\mathrm{(\ref{#1})}}}

\title{Gaussian Concentration bound for potentials satisfying Walters condition with subexponential continuity rates}

\author[1]{J.-R. Chazottes
\thanks{JRC benefited from an ECOS Nord project for his stay at San Luis Potos\'{\i}.}}

\author[2]{J. Moles
\thanks{JM benefited from an ECOS Nord project for his stay at Ecole Polytechnique.}}

\author[2]{E. Ugalde
\thanks{EU acknowledges Ecole Polytechnique for financial support (one-month stay).}}

\affil[1]{Centre de Physique Th\'eorique, Ecole Polytechnique, CNRS, 91128 Palaiseau, France.\newline
Email: \texttt{chazottes@cpht.polytechnique.fr}}

\affil[2]{Instituto de F\'{\i}sica, Universidad Aut\'onoma de San Luis Potos\'{\i}, S.L.P., 78290 M\'exico  \newline Emails: \texttt{molesjdn@ifisica.uaslp.mx}, \texttt{ugalde@ifisica.uaslp.mx }}

\date{Dated: \today}

\begin{document}

\maketitle

\begin{abstract}
We consider the full shift $T:\Omega\to\Omega$ where $\Omega=A^\N$, $A$ being a finite alphabet. For a class of
potentials which contains in particular potentials $\phi$ with variation decreasing like $O(n^{-\alpha})$ for some $\alpha>2$, we prove that their corresponding equilibrium state $\mu_\phi$ satisfies a Gaussian concentration bound. Namely, we prove that there exists a constant $C>0$ such that, for all $n$ and for all separately Lipschitz functions $K(x_0,\ldots,x_{n-1})$, the exponential moment of $K(x,\ldots,T^{n-1}x)-\int K(y,\ldots,T^{n-1}y)\, \dd\mu_\phi(y)$ is bounded by
$\exp\big(C\sum_{i=0}^{n-1}\Lip_i(K)^2\big)$. The crucial point is that $C$ is independent of $n$ and $K$. We then derive various consequences of this inequality. For instance, we obtain bounds on the fluctuations of the empirical frequency of blocks,
the speed of convergence of the empirical measure, and speed of Markov approximation of $\mu_\phi$. We also derive an almost-sure central limit theorem.
\end{abstract}

\newpage

\tableofcontents

\section{Introduction}

We consider the full shift $T:\Omega\to\Omega$ where $\Omega=A^\N$, $A$ being a finite alphabet. 
Given an ergodic measure $\mu$ on $\Omega$ and a continuous observable $f:\Omega\to\R$, we know by Birkhoff's ergodic
theorem that  $n^{-1}S_nf(x)$ converges, for $\mu$-almost every $x$, to $\int f\dd\mu$. (We use the standard notation $S_n f=f+f\circ T+\cdots +f\circ T^{n-1}$.) To refine this result, we need more assumptions on
$\mu$ and $f$. For instance, if $\mu=\mu_\phi$ is the equilibrium state for a Lipschitz potential $\phi:\Omega\to\R$ and $f$ is also a Lipschitz function, then the following central limit theorem holds:
\begin{equation}\label{intro-clt}
\lim_{n\to\infty}\mu_\phi\left( x: \frac{S_nf(x)-n\int f\dd\mu_\phi}{\sqrt{n}}\leq u\right)=
\frac{1}{\sigma\sqrt{2\pi}}\int_{-\infty}^u \e^{-\frac{\xi^2}{2\sigma^2}}\dd \xi
\end{equation}
for all $u\in\R$, where $\sigma^2=\sigma_f^2$ is the variance of the process $\{f(T^n x)\}_{n\geq 0}$ where $x$ is distributed according to $\mu_\phi$.
\footnote{ $\sigma_f^2=\int f^2\dd\mu_\phi-\Big(\int f\dd\mu_\phi\Big)^2+
2\sum_{\ell\geq 1} \left(\int f\cdot f\circ T^\ell \dd\mu_\phi-\Big(\int f\dd\mu_\phi\Big)^2 \right)$} 
This result says in essence that the fluctuations of $S_nf(x)-n\int f\dd\mu_\phi$ are with high probability of
order $\sqrt{n}$, when $n\to\infty$. Fluctuations of order $n$, referred to as `large deviations', are unlikely to appear.
Indeed, for instance one has
\begin{equation}\label{intro-ld}
\mu_\phi\left(x : \frac{1}{n}S_nf(x) \geq \int f\dd\mu_\phi +u\right)\asymp \e^{-n I(u+\int f\dd\mu_\phi)}
\end{equation}
where $u\geq 0$ and $I(u)\geq 0$ is the so-called `rate function' which is (strictly) convex, such that $I(\int f\dd\mu_\phi)=0$, and equal to $+\infty$ outside
a certain finite interval $(\ushort{u}_f,\bar{u}_f)$.\footnote{For two positive sequences $(a_n), (b_n)$, $a_n\asymp b_n$ means that
$\lim_n (1/n)\log a_n =\lim_n (1/n)\log b_n$.} Of course, both the central limit theorem and the large deviation asymptotics have been obtained for more general potentials, and for more general `chaotic' dynamical systems. For a fairly recent review on probabilistic properties of nonuniformly hyperbolic dynamical systems modeled by Young towers, we refer to \cite{chazottes_survey}. 

In this paper, we are interested in concentration inequalities which describe the fluctuations
of observables of the form $K(x,Tx,\ldots,T^{n-1}x)$ around their average. The only restriction on $K$ is that it has to be separately Lipschitz. By this we mean that, for all $i=0,\ldots,n-1$, there exists a constant $\Lip_i(K)$ with 
\[
|K(x_{0},\dots,x_{i},\dots,x_{n-1})-K(x_{0},\dots,x'_{i},\dots,x_{n-1})| \leq \Lip_i(K)\, d(x_{i},x'_{i}).
\]
for all points $x_0,\dots,x_i,\dots,x_{n-1},x'_{i}$ in $\Omega$, where $d$ is the usual distance on $\Omega$ (see \eqref{def-distance}).
So $K$ can be nonlinear and implicitly defined. Of course, such a class contains partial sums of Lipschitz functions, namely functions of the form $K(x_0,\ldots,x_{n-1})=f(x_0)+\cdots + f(x_{n-1})$ for which $\Lip_i(K)=\Lip(f)$ for all $i$.  
Beside considering very general observables, the other essential characteristics of concentration inequalities is that they are valid
for all $n$, contrarily to the above two results which are valid only in the limit $n\to\infty$. More precisely, we shall prove the following `Gaussian concentration bound'. There exists a constant $C$ such that, for all $n$ and for all separately Lipschitz functions
$K(x_0,\ldots,x_{n-1})$, we have
\begin{align}
\nonumber
& \int \exp\left( K\left(x,Tx,\dots,T^{n-1}x\right)\right) \dd\mu_\phi(x) \\
& \leq \exp\left( \int K\left(x,Tx,\dots,T^{n-1}x\right) \dd\mu_\phi(x)\right)
\exp\left(C \sum_{j=0}^{n-1} \Lip_j(K)^2\right)\,.
\label{intro-gcb}
\end{align}
The crucial point is that $C$ is independent of $n$ and $K$.
By a standard argument (see below), the previous inequality implies that for all $u>0$
\begin{align}
\nonumber
& \mu_\phi\left(x: K\left(x,Tx,\dots,T^{n-1}x\right)
-\int K\left(y,Ty,\dots,T^{n-1}y\right) \dd\mu_\phi(y)\geq u\right)\\
& \leq \exp\left(-\frac{u^{2}}{4C\sum_{i=0}^{n-1}\Lip_i(K)^{2}}\right)\,.
\label{intro-ci}
\end{align}
The Gaussian concentration bound \eqref{intro-gcb} is known for Lipschitz potentials \cite{2012ChazottesGouezel}.
We shall prove that it remains true for a large subclass of potentials $\phi$ satisfying Walters condition.
For instance, the bound holds for a potential whose variation is $O(n^{-\alpha})$ for some $\alpha>2$. 
The proof of our result relies on two main ingredients. First, we start with a classical decomposition of $K-\int K$ as a telescopic sum of
martingale differences. Second, we have to do a second telescoping to use Ruelle's Perrons-Frobenius operator. But we do not have a spectral
gap anymore as in \cite{2012ChazottesGouezel} (in the case of Lipschitz potentials). Instead, we use a result of V. Maume-Deschamps \cite{maume-phd}
based on Birkhoff cones.

We apply the Gaussian concentration bound  and its consequences, like \eqref{intro-ci}, to various observables. 
On the one hand, we obtain concentration bounds for previously studied observables. We get the same bounds but they are no
more limited to equilibrium states with Lipschitz potentials. On the other hand, we consider observables not considered before.
Even when $K(x,\ldots,T^{n-1}x)=S_nf(x)$, we get a non-trivial bound. We then obtain a control on the fluctuations of the empirical frequency of blocks $a^0,\ldots,a^{k-1}$ around $\mu([a^0,\ldots,a^{k-1}])$, uniformly in $a^0,\ldots,a^{k-1}\in A^k$.
We then consider an estimator of the entropy $\mu_\phi$ based on hitting times.
The next application is about the speed of convergence of the empirical measure $(1/n)\sum_{i=0}^{n-1} \delta_{T^i x}$ towards $\mu_\phi$ in Wasserstein distance. Then we obtain an upper bound for the $\bar{d}$-distance between any shift-invariant probability measure and $\mu_\phi$. This distance is bounded by the square root of their relative entropy, times a constant. A consequence of this inequality is a bound for the speed of convergence of the Markov approximation of $\mu_\phi$ in $\bar{d}$-distance. Then we quantify the `shadowing' of an orbit by another one which has to start in a subset of
$\Omega$ with $\mu_\phi$-measure $1/3$, say. Finally, we prove an almost-sure version of the central limit theorem. This application shows in particular that concentration inequalities can also be used to obtain limit theorems.

\section{Setting and preliminary results}

Let $\Omega=A^\N$ where $A$ is a finite set. We denote by $x=x^0x^1\dots$ the elements of
$\Omega$ (hence $x^i\in A$), and by $T$ the shift map: $(Tx)^k=x^{k+1}$, $k\in\N$. (We use upper indices instead of lower indices because we will need to consider bunches of
points in $\Omega$, {\em e.g.}, $x_0,x_1,\dots,x_p$, $x_i\in\Omega$.) We use the classical distance
\begin{equation}\label{def-distance}
d_\theta(x,y)=\theta^{\inf\{k: x^k\neq y^k\}}
\end{equation}
where $\theta\in(0,1)$ is some fixed number. Probability measures are defined on the Borel sigma-algebra of $\Omega$ which is generated by cylinder
sets. Let $\phi:\Omega\to\R$ be a continuous potential, which means that 
\[
\var_n (\phi):=\sup\lbrace|\phi(x)-\phi(y)| : x^i=y^i, 0\leq i\leq n-1 \rbrace\xrightarrow[]{n\to\infty} 0.
\]
The sequence $(\var_n (\phi))_{n\geq 1}$ is the modulus of continuity of $\phi$ and it is called the `variation' of $\phi$ in our context.
By the way, we denote by $\mathscr{C}(\Omega)$ the Banach space of real-valued continuous functions on $\Omega$ equipped with the supremum norm
$\|\cdot\|_\infty$.
We put further restrictions on $\phi$, namely that it must satisfy the Walters condition \cite{1978Walters}. For $x,y$ in $\Omega$ let
\[
W(\phi,x,y)=\sup_{n\in\N}\sup_{a \in A^n}\left| S_n\phi(ax)-S_n\phi(ay)\right|\,.
\]
We assume that $W(\phi,x,y)$ exists and that there exists $W(\phi)>0$ such that
\begin{equation}\label{Wphi}
\sup_{x,y\, \in\,\Omega}W(\phi,x,y)\leq W(\phi)\,.
\end{equation}
Now for $p\in\N$ let
\[
W_p(\phi):=\sup\{W(\phi,x,y): x^i=y^i, 0\leq i\leq p-1\}\,.
\]
\begin{definition}\label{def-W-cond}
$\phi$ is said to satisfy Walters' condition if $(W_p(\phi))_{p\in\N}$ is a strictly positive sequence
and decreases to $0$ as $p\to\infty$.
\end{definition}

We now make several remarks on Walters' condition.
First, observe that locally constant potentials do not satisfy this condition because $W_p(\phi)=0$ for all $p$ larger than some $p_0$. But one can in fact work with any 
strictly positive sequence $(\widetilde{W}_p(\phi))_{p\in\N}$ decreasing to zero such that $W_p(\phi)\leq \widetilde{W}_p(\phi)$ for all $p$, {\em e.g.}, $\max(W_p(\phi),\eta^p)$ for some fixed $\eta\in(0,1)$.
Second, one easily checks that 
\begin{equation}\label{cphivar}
\var_{p+1}(\phi) \leq W_p(\phi)\leq \sum_{k=p+1}^\infty \var_k(\phi)\,, \, p\in\N.
\end{equation}
Hence the set of potentials satisfying Walters' condition contains the set of potentials with summable variation. In particular, 
$(W_p(\phi))_p$ is bounded above by a geometric sequence if and only if $(\var_p(\phi))_p$ is also bounded above by a geometric sequence.
This corresponds to the case of Lipschitz or H\"older potentials (with respect to $d_\theta$).

Now define Ruelle's Perron-Frobenius operator $P_\phi:\mathscr{C}(\Omega)\to\mathscr{C}(\Omega)$ as 
\[
P_\phi f(x)=\sum_{Ty=x} f(y)\e^{\phi(y)}\,.
\]
The next step is to define a function space preserved by $P_\phi$ and on which it has good spectral properties. We take the space of Lipschitz functions with respect to a new distance $d_\phi$ built out of $\phi$ as follows. 
\begin{definition}[The distance $d_\phi$]
For $x,y\in\Omega$ let
\[
d_\phi(x,y)=W_p(\phi) \quad\textup{if} \quad d_\theta(x,y)=\theta^{p}
\]
and $d_\phi(x,x)=0$. 
\end{definition}
Now define
\[
\mathcal{L}_\phi=\lbrace f\in\mathscr{C}(\Omega) : \exists M>0 \;\textup{such that}\;\var_{n}(f)\leq MW_n(\phi), n=1,2,\ldots \rbrace
\]
and
\begin{equation*}
%\label{def-lip-phi}
\Lip_{\phi} (f)
=
\sup\left\{\frac{|f(x)-f(y)|}{d_\phi(x,y)}: x\neq y\right\}
=
\sup\left\{\frac{\var_{n}(f)}{W_n(\phi)}: n\in\N\right\}.
\end{equation*}
One can then define a norm on $\mathcal{L}_\phi$, making it a Banach space, by setting
\begin{equation*}
%\label{def-Lphi-norm}
\|f\|_{\mathcal{L}_\phi}=\|f\|_{\infty}+\Lip_{\phi} (f).
\end{equation*}

\begin{remark}
The usual Banach space of Lipschitz functions is defined as follows. Let
\[
\mathcal{L}_\theta=\lbrace f\in\mathscr{C}(\Omega) : \exists M>0 \;\textup{such that}\;\var_{n}(f)\leq M\theta^n, n=1,2,\ldots \rbrace
\]
and
\begin{equation*}
%\label{def-lip}
\Lip_\theta(f)
=\sup\left\{\frac{|f(x)-f(y)|}{d_\theta(x,y)}: x\neq y\right\}
=\sup\left\{\frac{\var_{n}(f)}{\theta^n}: n\in\N\right\}.
\end{equation*}
The canonical norm making $\mathcal{L}_\theta$ a Banach space is $\|f\|_{\mathcal{L}_\theta}=\|f\|_\infty + \Lip_\theta(f)$.

In view of \eqref{cphivar}, if we have $W_n(\phi)=O(\theta^n)$, then $\mathcal{L}_\theta=\mathcal{L}_\phi$.
If we now have, for instance, $W_n(\phi)=O(n^{-q})$ for some $q>0$, then we get a bigger space which contains in particular all functions $f$ such that $\var_n(f)=O(n^{-r})$ with $r\geq q$.
\end{remark}

The following result is instrumental to this article. In brief, it tells us that a potential $\phi$ satisfying Walters' condition has a unique equilibrium state, which will be denoted by $\mu_\phi$, and gives a speed of convergence for the properly normalized iterates of the associated Ruelle's Perron-Frobenius operator.
The first part of the theorem is due to Walters, while the second one is due to Maume-Deschamps and can be found in her PhD thesis \cite[Chapter I.2]{maume-phd}. Unfortunately, her result was not published even though it is much sharper than the result in \cite{1997MaumeKondahSchmitt}.

\begin{theorem}[\cite{1978Walters}, \cite{maume-phd}]\label{backgroundthm}
Let $\phi:\Omega\to\R$ satisfying Walters' condition as above. Then the following holds.
\begin{itemize}
\item[A.]
There exists a unique triplet $(h_\phi, \lambda_\phi,\nu_\phi)$ such that
$h_\phi\in \mathcal{L}_\phi$ and is strictly positive, $\|\log h_\phi\|_\infty<\infty$, $\lambda_\phi>0$, $\nu_\phi$ a fully supported probability measure such 
that $\int h_\phi\, \dd\nu_\phi=1$. Moreover, $P_\phi h_\phi=\lambda_\phi h_\phi$ and $P_\phi^* \nu_\phi=\lambda_\phi \nu_\phi$, and $\phi$ has a unique 
equilibrium state $\mu_\phi=h_\phi \nu_\phi$ which is mixing. \footnote{This means that for any pair of cylinders $B,B'$
$\lim_{n\to\infty} \mu_\phi(B\cap T^{-n}B')=\mu_\phi(B)\mu_\phi(B')$. In particular $\mu_\phi$ is ergodic.}
\item[B.]
There exists a positive sequence $(\epsilon_n)_{n\in\N}$ converging to zero, such that, for any $f\in \mathcal{L}_\phi$,
\begin{equation}\label{ineq-vero}
\left\| \frac{P_\phi^n f}{\lambda_\phi^n}-h_\phi \int f \dd\nu_\phi\right\|_\infty \leq \cte{backgroundthm}\, \epsilon_n 
\|f\|_{\mathcal{L}_\phi}, \;\forall n\in\N\,.
\end{equation}
Morover, one has the following behaviors:
\begin{enumerate}
\item
If $W_n(\phi)=O(\eta^n)$ for some $\eta\in (0,1)$, then there exists $\eta'\in (0,1)$ such that $\epsilon_n=O({\eta'}^n)$.
\item
If $W_n(\phi)=O(n^{-\alpha})$ for some $\alpha>0$, then $\epsilon_n=O(n^{-\alpha})$.
\item
If $W_n(\phi)=O(\theta^{(\log n)^{\alpha}})$ for some $\theta\in (0,1)$ and $\alpha>1$, then, for any $\epsilon>0$, 
$\epsilon_n=O(\theta^{(\log n)^{\alpha-\epsilon}})$.
\item 
If $W_n(\phi)=O(\e^{-cn^\alpha})$ for some $c>0$ and $\alpha\in (0,1)$, then there exists $c'>0$ such that 
$\epsilon_n=O\big(\e^{-c'n^{\frac{\alpha}{\alpha+1}}}\big)$.
\end{enumerate}
\end{itemize}
\end{theorem}

The fact that $\mu_\phi$ is an equilibrium state means that it maximizes the functional $\mu\mapsto h(\nu)+\int \phi\,\dd\nu$
over the set of shift-invariant probability measures on $\Omega$, where $h(\nu)$ is the entropy of $\nu$, and the maximum
is equal to the topological pressure $P(\phi)$ of $\phi$ (see {\em e.g.} \cite{1998Keller}), and we have $P(\phi)=\log \lambda_\phi$.

Let us give examples of potentials. First consider $A=\{-1,1\}$ and $p>1$, and define
\[
\phi(x)=-\sum_{n\geq 2} \frac{x^0x^{n-1}}{n^p}\,.
\]
One can check that $W_n(\phi)=O(n^{-p+2})$. This is the analog of the so-called long-range Ising model on $\mathds{N}$.
Let us now take $A=\{0,1\}$ and let $[0^k1]=\{x\in\Omega : x^i=0, 0\leq i\leq k-1,\,\text{and}\, x_k=1\}$. Let $(v_n)$ be
a monotone decreasing sequence of real numbers converging to $0$ and define
\[
\phi(x)=
\begin{cases}
v_k & \text{if} \; x\in [0^k1]\\
0 & \text{if} \; x=(0,0,\ldots) \\
0 & \text{otherwise}\,.
\end{cases}
\]
One can check that $\var_n(\phi)=v_n$. This example is taken from \cite{pollicott2000}.

\begin{remark}\label{rem-uem}
Let us briefly explain how we can interpret an equilibrium state for a non Lipschitz potential as an absolutely continuous invariant measure of a piecewise
expanding map of the unit interval with a Markov partition.
It is well-known that a uniformly expanding map $S$ of the unit interval with a finite Markov partition which is piecewise
$C^{1+\eta}$, for some $\eta>0$, can be coded by a subshift of finite type $(\Omega,T)$ over a finite alphabet. Then, $-\log |S'|$ induces a potential $\phi$ 
on $\Omega$ which is Lipschitz (with respect to $d_\theta$). The pullback of $\mu_\phi$ is then the unique absolutely continuous invariant probability 
measure for $S$. In \cite{2012ColletGalves}, the authors showed that, given $\phi$ which is not Lipschitz,
one can construct a uniformly expanding map of the unit interval with a finite Markov partition which is
piecewise $C^1$, but not piecewise $C^{1+\eta}$ for any $\eta>0$, and such that the pullback of $\mu_\phi$ is the Lebesgue measure.
\end{remark}

\section{Main result and applications}

\subsection{Gaussian concentration bound}

We can now state our main theorem whose proof is deferred to Section \ref{sec:proof-main-thm}. We start by the definition
of separately $d_\theta$-Lipschitz functions.

\begin{definition}\label{def-sep-dphi-lip}
A function $K:\Omega^n\to\R$ is said to be separately $d_\theta$-Lipschitz if, for all $i$,
there exists a constant $\Lip_{\theta,i}(K)$ with 
\[
|K(x_{0},\dots,x_{i},\dots,x_{n-1})-K(x_{0},\dots,x'_{i},\dots,x_{n-1})| \leq \Lip_{\theta,i}(K)\, d_\theta(x_{i},x'_{i}).
\]
for all points $x_0,\dots,x_i,\dots,x_{n-1},x'_{i}$ in $\Omega$.
\end{definition}

\begin{theorem}\label{maintheo}
Suppose that $\phi$ satisfies one of the following conditions:
\begin{enumerate}
\item
$W_n(\phi)=O(\theta^n)$ (that is, $\phi$ is $d_\theta$-Lipschitz);
\item
$W_n(\phi)=O(n^{-\alpha})$ for some $\alpha>1$;
\item
$W_n(\phi)=O(\theta^{(\log n)^{\alpha}})$ for some $\theta\in (0,1)$ and $\alpha>1$; 
\item
$W_n(\phi)=O(\e^{-cn^\alpha})$ for some $c>0$ and $\alpha\in (0,1)$.
\end{enumerate} 
Then the process $(x,Tx,\ldots)$, with $x$ distributed according to $\mu_\phi$, satisfies the following Gaussian concentration bound. There exists $\cte{maintheo} >0$ such that for any $n\in\N$ and for any separately $d_\theta$-Lipschitz function $K:\Omega^n\to\R$, we have
\begin{align}
\nonumber
& \int \exp\left( K\left(x,Tx,\dots,T^{n-1}x\right)\right) \dd\mu_\phi(x) \\
& \leq \exp\left( \int K\left(x,Tx,\dots,T^{n-1}x\right) \dd\mu_\phi(x)\right)
\exp\left(\cte{maintheo} \sum_{j=0}^{n-1} \Lip_{\theta,j}(K)^2\right)\,.
\label{gcb}
\end{align}
\end{theorem}
Three remarks are in order.
First, we conjecture that this theorem is valid under the condition $\sum_n \var_n(\phi)<\infty$.
Second, it would be useful to have an explicit formula for $\cte{maintheo}$ in \eqref{gcb}. Unfortunately, this constant is proportional to 
$\cte{backgroundthm}$ (see Theorem \ref{backgroundthm}) which is cumbersome since it involves the eigendata of $P_\phi$ .
Third, for the sake of simplicity, we considered the full shift $A^{\N}$. In fact, our results remain true if $\Omega\subset A^{\N}$ is a topologically mixing one-sided subshift of finite type. Moreover, one can extend Theorem \ref{maintheo} to bilateral subshifts of finite type by a trick used in \cite{2012ChazottesGouezel}.

We now give some corollaries of our main theorem that we will be used in the section on applications.
First, by \eqref{cphivar} we immediately obtain the following corollary.
\begin{corollary}
If there exists $\alpha>2$ such that 
\[
\var_n (\phi) =O\left(\frac{1}{n^\alpha}\right)
\]
then we have the Gaussian concentration bound \eqref{gcb}.
\end{corollary}

Next, we get the following concentration inequalities from \eqref{gcb}.
\begin{corollary}\label{coromaintheo}
For all $u> 0$, we have
\begin{align}
\nonumber
& \mu_\phi\left(x\in\Omega: K\left(x,Tx,\dots,T^{n-1}x\right)
-\int K\left(y,Ty,\dots,T^{n-1}y\right) \dd\mu_\phi(y)\geq u\right)\\
& \leq \exp\left(-\frac{u^{2}}{4\cte{maintheo}\sum_{i=0}^{n-1}\Lip_{\theta,i}(K)^{2}}\right)
\label{dev-gcb}
\end{align}
and
\begin{align}
\nonumber
& \mu_\phi\left(x\in\Omega:\left|K\left(x,Tx,\dots,T^{n-1}x\right)
-\int K\left(y,Ty,\dots,T^{n-1}y\right) \dd\mu_\phi(y)\right|\geq u\right)\\
& \leq 2\exp\left(-\frac{u^{2}}{4\cte{maintheo}\sum_{i=0}^{n-1}\Lip_{\theta,i}(K)^{2}}\right)\,.
\label{dev-abs-gcb}
\end{align}
\end{corollary}
\begin{proof}
Inequality \eqref{dev-gcb} follows by a well-known trick referred to as Chernoff's bounding method \cite{boucheron_lugosi_massard}. Let us give the proof for completeness.
Let $u>0$. For any random variable $Y$, Markov's inequality tells us that
$\proba(Y\geq u)\leq \e^{-\xi u} \E\left(\e^{\xi Y}\right)$ for all $\xi>0$.
Now let 
\[
Y=K\left(x,Tx,\dots,T^{n-1}x\right)-\int K\left(y,Ty,\dots,T^{n-1}y\right) \dd\mu_\phi(y)\,.
\]
Using \eqref{gcb} and optimizing over $\xi$, we get \eqref{dev-gcb}. Inequality \eqref{dev-abs-gcb} follows
by applying \eqref{dev-gcb} to $-K$ and then summing up the two bounds.
\end{proof}
The last corollary we want to state is about the variance of any separately $d_\theta$-Lipschitz function.
\begin{corollary}
We have
\begin{align}
\nonumber
\label{devroye}
& \int \left(K\left(x,Tx,\dots,T^{n-1}x\right)
-\int K\left(y,Ty,\dots,T^{n-1}y\right) \dd\mu_\phi(y)\right)^2 \dd\mu_\phi(x)\\
& \leq 2\cte{maintheo} \sum_{i=0}^{n-1}\Lip_{\theta,i}(K)^{2}\,.
\end{align}
\end{corollary}
\begin{proof}
To alleviate notations, we simply write $K$ instead of $K\left(x,Tx,\dots,T^{n-1}x\right)$, $\int K$ instead of 
$\int K\left(y,Ty,\dots,T^{n-1}y\right) \dd\mu_\phi(y)$, and so on and so forth.
Applying \eqref{gcb} to $\xi K$ where $\xi$ is any real number different from $0$, we get
\[
\int \exp\left(\xi \Big(K-\int K\Big)\right)\leq \exp\left(\cte{maintheo} \xi^2 \sum_{j=0}^{n-1} \Lip_{\theta,j}(K)^2\right)\,.
\]
Now by Taylor expansion we get
\[
1+\frac{\xi^2}{2}\int\Big(K-\int K\Big)^2 + o(\xi^2)
\leq 1+\cte{maintheo} \xi^2 \sum_{j=0}^{n-1} \Lip_{\theta,j}(K)^2 +o(\xi^2)\,.
\]
Dividing by $\xi^2$ on both sides and then taking the limit $\xi\to 0$, we obtain the desired inequality.
\end{proof}

Although we were not able to prove the Gaussian concentration bound for separately $d_\phi$-Lipschitz functions,
for many applications separately $d_\theta$-Lipschitz functions are more natural. Furthermore there is
a notable class of separately $d_\phi$-Lipschitz functions, namely Birkhoff sums of the potential itself, for which
our theorem holds. Indeed, when $\phi\in \mathcal{L}_\phi$, 
the function $K(x,\ldots,T^{n-1}x)=S_n \phi(x)$ is obviously separately $d_\phi$-Lipschitz and 
$\Lip_{\phi,j}(K)=\Lip_\phi(\phi)$ for all $j$.  We have the following result.

\begin{theorem}\label{GCBBirkhoff}
Under the hypotheses of Theorem \ref{maintheo}, there exists $\cte{GCBBirkhoff}>0$ such that, for any $\psi\in \mathcal{L}_\phi$, for all $u> 0$, and for all $n\in\N$, we have
\begin{align}
\nonumber
& \mu_\phi\left(x\in\Omega: \frac{1}{n}S_n\psi(x)
- \int \psi\, \dd\mu_\phi\geq u\right)\\
& \leq \exp\left(-\frac{n u^{2}}{4\cte{GCBBirkhoff}\Lip_{\phi}(\psi)^{2}}\right)\,.
\label{dev-gcb-birkhoff}
\end{align}
\end{theorem}

The proof is left to the reader. The main (simple) modification lies in the proof of Lemma \ref{lem_controle_Kp} in which considering a Birkhoff sum of a
$d_\phi$-Lipschitz function works fine, whereas we are stuck for a general separately $d_\phi$-Lipschitz function.

We will apply this result with $\psi=-\phi$ to derive concentration bounds for hitting times.
Note that under the assumptions of this theorem, $\{\psi(T^n x)\}_{n\geq 0}$ satisfies the central limit theorem \cite[Chapter 2]{maume-phd}.

\subsection{Related works}

The novelty here is to prove a Gaussian concentration bound for potentials with a variation decaying subexponentially. 
For $\phi$ is Lipschitz, Theorem \ref{maintheo} was proved in \cite{2012ChazottesGouezel}. 
The main goal of \cite{2012ChazottesGouezel} was then to deal with nonuniformly hyperbolic systems modeled by a Young
tower. 
For a tower with a return-time to the base with exponential tails, the authors of \cite{2012ChazottesGouezel} proved a Gaussian concentration 
bound. For polynomial tails, they proved moment concentration bounds. 
For $C^{1+\eta}$ maps of the unit interval with an indifferent fixed point, which are thus nonuniformly expanding, we are in the latter situation. In view of 
Remark \ref{rem-uem} above, we deal here with maps whose derivative is not H\"older continuous, but which are still uniformly expanding.

Let us also mention the paper \cite{2014GalloTakahashi} in which the authors prove a Gaussian concentration bound for 
$\phi$ of summable variation (whereas we need a bit more than summable). Their proof is based on coupling. However, they consider functions $K$ on 
$A^n$, not on $\big(A^\N\big)^n=\Omega^n$ as in this paper. For such functions, the analogue of $\Lip_{\theta,i}(K)$ is $\delta_i(K)=\sup\{|K(a^0,\ldots,a^i,
\ldots, a^{n-1})-K(b^0,\ldots,b^i,\ldots, b^{n-1})|: a^j=b^j, \forall j\neq i\}$.
It is clear that a Gaussian concentration bound for functions $K:\big(A^\N\big)^n\to\R$ implies a Gaussian concentration bound
for functions $K:A^n\to\R$, but the converse is not true.

\subsection{Applications}

We now give several applications of the Gaussian concentration bound \eqref{gcb} and its corollaries. Throughout this section, $\mu_\phi$ is the equilibrium state for a potential $\phi$ satisfying one of the conditions {\em 1-4} in Theorem \ref{maintheo}.

\subsubsection{Birkhoff sums}

Let $f:\Omega\to\R$ be a $d_\theta$-Lipschitz function and define
\[
K(x_0,\ldots,x_{n-1})=f(x_0)+\cdots+f(x_{n-1})
\]
whence $K(x,Tx,\ldots, T^{n-1}x)=f(x)+f(Tx)+\cdots+f(T^{n-1}x):=S_nf(x)$ is the Birkhoff sum of $f$.
Clearly, $\Lip_{\theta,i}(K)=\Lip_\theta(f)$ for all $i=0,\ldots, n-1$.
Applying Corollary \ref{coromaintheo} we immediately get
\begin{equation}\label{gcbbirkhoffsums}
\mu_\phi\left(x : \left|\frac{S_nf(x)}{n}-\int f \dd\mu_\phi\right|\geq u\right)
\leq 2 \exp\left(- \pcte{gcbbirkhoffsums} n u^2\right)
\end{equation}
for all $n\geq 1$ and $u\in\R_+$, where 
\[
\pcte{gcbbirkhoffsums}=\frac{1}{4\cte{maintheo}\Lip_\theta(f)^2}\,.
\]
This bound can be compared with the large deviation asymptotics \eqref{intro-ld}. We see that it has the right behavior in $n$.
Replacing $u$ by $u/\sqrt{n}$ in \eqref{gcbbirkhoffsums} we get
\[
\mu_\phi\left(x : \left|S_nf(x)-n\int f \dd\mu_\phi\right|\geq u\sqrt{n}\right)
\leq 2 \exp\left(- \pcte{gcbbirkhoffsums} u^2\right)
\]
for all $n$ and $u>0$. This can be compared with the central limit theorem \eqref{intro-clt}. We can see that the previous bound is consistent with that theorem. Note that the central limit is about convergence in law, whereas here we obtain a (non-asymptotic) bound from which one cannot deduce a convergence in law.

\subsubsection{Empirical frequency of blocks}

Take $f(x)=\un_{[a^{0,k-1}]}(x)$ where
\[
[a^{0,k-1}]=\{x\in\Omega : x^i=a^i,i=0,\ldots,k-1\}
\]
is a given $k$-cylinder.
Let
\[
\mathfrak{f}_n(x,a^{0,k-1})=\frac{\sum_{k=0}^{n-k} \un_{[a^{0,k-1}]}(T^k x)}{n-k+1}\,.
\] 

This is the `empirical frequency' of the block $a^{0,k-1}\in A^k$ in the orbit of $x$ up to time $n-k$. 
By Birkhoff's ergodic theorem, we know that, for each $a^{0,k-1}$, $\mathfrak{f}_n(x,a^{0,k-1})$ goes to
$\mu_\phi([a^{0,k-1}])$ for $\mu_\phi$-almost all $x$. The next theorem quantifies this asymptotic statement. Notice
that we can control the fluctuations of $\mathfrak{f}_n(x,a^{0,k-1})$ around $\mu_\phi([a^{0,k-1}])$ uniformly in
$a^{0,k-1}$.
\begin{theorem}
For all $n\in\N$, for all $1\leq k\leq n$ and for all $u>0$ we have
\[
\mu_\phi\left(x : \max_{a^{0,k-1}}\left|\mathfrak{f}_n(x,a^{0,k-1})-\mu_\phi([a^{0,k-1}])\right| 
 \geq \frac{(u + c\sqrt{k}\,)\, \theta^{-k}}{\sqrt{n-k+1}}\,\right) \\
\leq \e^{-\frac{u^2}{4\cte{maintheo}}}
\]
where $c=2\sqrt{2\cte{maintheo}\log |A|}$. Moreover, if $k=k(n)=\zeta \log n$ for some $\zeta>0$, then 
\begin{align*}
& \mu_\phi\left(x : \max_{a^{0,k(n)-1}}\left|\mathfrak{f}_n(x,a^{0,k(n)-1})-\mu_\phi([a^{0,k(n)-1}])\right| 
 \geq \frac{(u + c'\sqrt{\log n}\,) n^{\zeta |\log\theta|}}{\sqrt{n-k(n)+1}}\,\right) \\
& \leq \e^{-\frac{u^2}{4\cte{maintheo}}}
\end{align*}
where $c'=2\sqrt{2\zeta \cte{maintheo}\log |A|}$.
\end{theorem}

\begin{proof}
Define the function $K:\Omega^{n-k+1}\to\R$ by
\[
K(x_0,\ldots,x_{n-k})=\max_{a^{0,k-1}} Z(a^{0,k-1};x_0,\ldots,x_{n-k})
\]
where
\[
Z(a^{0,k-1};x_0,\ldots,x_{n-k})=\left|\frac{\sum_{j=0}^{n-k} \un_{[a^{0,k-1}]}(x_j)}{n-k+1}-\mu_\phi([a^{0,k-1}])\right|\,.
\]
It is left to the reader to check that $\Lip_{\theta,j}(K)=\frac{\Lip_\theta(f)}{n-k+1}=\frac{1}{\theta^{k}(n-k+1)}$, so we get immediately from \ref{gcbbirkhoffsums} 
\begin{align*}
& \mu_\phi\left(x\in\Omega : K(x,Tx,\ldots, T^{n-k}x) \geq u+\int K\left(y,Ty,\dots,T^{n-k}y\right) \dd\mu_\phi(y)\right)\\
& \leq \exp\left(-\frac{\theta^{2k}}{4\cte{maintheo}} \, (n-k+1) u^2\right)
\end{align*}
for all $n\geq 1$ and $u>0$. To complete the proof, we need a good upper bound for $\int K\left(y,Ty,\dots,T^{n-k-1}y\right) \dd\mu_\phi(y)$. Actually, this can be
done by using again the Gaussian concentration bound.
Using \eqref{gcb} and Jensen's inequality we get for any $\xi>0$
\begin{align*}
& \exp\left(\xi \int K\left(x,Tx,\dots,T^{n-k}x\right) \dd\mu_\phi(x) \right)\\
& \leq \int \exp\left( \xi \max_{a^{0,k-1}} Z(a^{0,k-1};x,Tx,\dots,T^{n-k}x)\right) \dd\mu_\phi(x)\\
& \leq \sum_{a^{0,k-1}\in A^k}  \int \exp\left( \xi Z(a^{0,k-1};x,Tx,\dots,T^{n-k}x)\right) \dd\mu_\phi(x)\\
& \leq 2 |A|^k \, \exp\left(\frac{\cte{maintheo} \theta^{-2k} \xi^2}{n-k+1}\right)\,.
\end{align*}
The third inequality is obtained by using the trivial inequality
\[
\e^{\max_{i=1}^{p}a_i} \leq \sum_{i=1}^p \e^{a_i}\,.
\]
Taking logarithms on both sides and then dividing by $\xi$, we have
\[
\int K\left(x,Tx,\dots,T^{n-k}x\right) \dd\mu_\phi(x) \leq \frac{\log 2+k\log |A|}{\xi} + \frac{\cte{maintheo} \theta^{-2k} \xi}{n-k+1}\,.
\]
There is a unique $\xi>0$ minimizing the right-hand side, hence
\[
\int K\left(x,Tx,\dots,T^{n-k}x\right) \dd\mu_\phi(x) \leq 2\theta^{-k} \sqrt{\frac{\cte{maintheo}(k+1)\log |A|}{n-k+1}}
\]
where we used that $\log 2\leq \log|A|$.
Hence we get the desired estimate.
\end{proof}

Note that $\log|A|$ is the topological entropy of the full shift with alphabet $A$.

\subsubsection{Hitting times and entropy}

For $x,y\in\Omega$, let
\[
T_{x^{0,n-1}}(y)=\inf\{j\geq 1 : y^{j,j+n-1}=x^{0,n-1}\}\,.
\]
This is the first time that the $n$ first symbols of $x$ appear in $y$. 
We assume that $\phi$ satisfies 
\begin{equation}\label{bibi}
\var_n(\phi)=O\left(\frac{1}{n^\alpha}\right)\quad\text{for some}\;\alpha>2\,.
\end{equation}
One can prove (see \cite{chazottes-ugalde-2005}) that
\[
\lim_{n\to\infty}\frac{1}{n}\log T_{x^{0,n-1}}(y)=h(\mu_\phi)\,,\quad \textup{for}\;\mu_\phi\otimes\mu_\phi\textup{-almost every}\, (x,y)\,.
\]
Roughly, this means that, if we pick $x$ and $y$ independently, each one according to $\mu_\phi$, then the time it takes to see the first $n$ symbols of $x$ appearing in $y$ for the first time is $\approx \e^{n h(\mu_\phi)}$. 

\begin{theorem}
If $\phi$ satisfies \eqref{bibi}, then there exist strictly positive constants $c_1,c_2$ and $u_0$ such that, for all $n$ and
for all $u>u_0$,
\[
(\mu_\phi\otimes\mu_\phi)\Big\{(x,y): \frac{1}{n}\log T_{x^{0,n-1}}(y)\geq  h(\mu_\phi)+u\Big\}\leq c_1 \e^{-c_2 nu^2}
\]
and
\[
(\mu_\phi\otimes\mu_\phi)\Big\{(x,y): \frac{1}{n}\log T_{x^{0,n-1}}(y)\leq  h(\mu_\phi)-u\Big\}\leq c_1 \e^{-c_2 nu}\,.
\]
\end{theorem}
These bounds were obtained in \cite{2011ChazottesMaldonado} when $\phi$ is Lipschitz. Observe that the probability of
being above $h(\mu_\phi)$ is bounded above by $c_1 \e^{-c_2 nu^2}$, whereas the probability of being below $h(\mu_\phi)$
is bounded above by $c_1 \e^{-c_2 nu}$. The proof of this theorem being very similar to that given in \cite{2011ChazottesMaldonado}, we omit the details and only sketch it.  We cannot directly deal with $T_{x^{0,n-1}}(y)$ but we have $\log T_{x^{0,n-1}}(y)=\log \big(T_{x^{0,n-1}}(y)\mu_\phi([x^{0,n-1}])\big)-\log \mu_\phi([x^{0,n-1}])$. Then we
use Theorem \ref{GCBBirkhoff} for $\psi=-\phi$, assuming (without loss of generality) that $P(\phi)=0$, that is,
$h(\mu_\phi)=-\int \phi \, \dd\mu_\phi$, because we can control uniformly in $x$ the approximation
$-\log \mu_\phi([x^{0,n-1}])\approx S_n(-\phi)(x)$. To control the other term, we use that the law of $T_{x^{0,n-1}}(y)\mu_\phi([x^{0,n-1}])$ is well approximated by an exponential law.

Another estimator of $h(\mu_\phi)$ is the so-called plug-in estimator. We could also obtain concentration bounds for it in the spirit
of \cite{2011ChazottesMaldonado}.

\subsubsection{Speed of convergence of the empirical measure}

Instead of looking at the frequency of a block $a_1^k$ we can consider a global object, namely the empirical measure
\[
\mathscr{E}_n(x)=\frac{1}{n}\sum_{j=0}^{n-1} \delta_{T^j x}\,.
\]
For $\mu_\phi$-almost every $x$, we know that
\[
\frac{1}{n}\sum_{j=0}^{n-1} \delta_{T^j x} \xrightarrow[]{n\to\infty} \mu_\phi
\]
where the convergence is in the weak topology on the space of probability measures $\mathscr{M}(\Omega)$ on $\Omega$. This is a consequence of Birkhoff's ergodic theorem. The natural question is: how fast does this convergence takes place? We can answer this question by using the Kantorovich distance $d_{{\scriptscriptstyle K}}$ which metrizes weak topology on $\mathscr{M}(\Omega)$:
\[
d_{{\scriptscriptstyle K}}(\nu_1,\nu_2) =
\sup\left\{ \int f \dd\nu_1-\int f \dd\nu_2 : f:\Omega\to\R\;\text{such that}\;\Lip_\theta(f)=1 \right\}\,.
\]
We have the following result.
\begin{theorem}
For all $u>0$ and all $n\geq 1$ we have
\begin{equation}\label{concentrationkanto}
\mu_\phi\left( x: 
\big| d_{{\scriptscriptstyle K}}(\mathscr{E}_n(x),\mu_\phi)-\int d_{{\scriptscriptstyle K}}(\mathscr{E}_n(y),\mu_\phi)\, \dd\mu_\phi(y)\big|
\geq u\right)
\leq 2\, \e^{-\pcte{concentrationkanto}n u^2}
\end{equation}
where $\pcte{concentrationkanto}=(4\cte{maintheo})^{-1}$.
\end{theorem}
\begin{proof}
Let 
\[
K(x_0,\ldots,x_{n-1})=
\sup\left\{ \frac{1}{n}\sum_{j=0}^{n-1} f(x_j)-\int f \dd\mu_\phi : f:\Omega\to\R\;\text{with}\;\Lip_\theta(f)=1 \right\}\,.
\]
Of course, $K(x,Tx,\ldots,T^{n-1}x)=d_{{\scriptscriptstyle K}}(\mathscr{E}_n(x),\mu_\phi)$. It is left to the reader to check that
\[
\Lip_{\theta,i}(K)\leq \frac{1}{n}\,,\, i=0,\ldots,n-1\,.
\]
The result follows at once by applying inequality \eqref{dev-abs-gcb}.
\end{proof}

It is natural to ask for a good upper bound for $\int d_{{\scriptscriptstyle K}}(\mathscr{E}_n(y),\mu_\phi) \dd\mu_\phi(y)$ because 
this would give a control on the fluctuations of $d_{{\scriptscriptstyle K}}(\mathscr{E}_n(x),\mu_\phi)$ around $0$. Getting such a bound turns out to be difficult. In \cite[Section 8]{chazottes-collet-redig2017} it is proved that 
\[
\int d_{{\scriptscriptstyle K}}(\mathscr{E}_n(y),\mu_\phi)\, \dd\mu_\phi(y) \preceq \frac{1}{n^{\frac{1}{2(1+\log|A|)}}}\,.
\]
For two positive sequences $(a_n), (b_n)$, $a_n \preceq b_n$ means that $\limsup_n \frac{\log a_n}{\log b_n}\leq 1$. One could in principle get a non-asymptotic but messy bound. 

\subsubsection{Relative entropy, $\bar{d}$-distance and speed of Markov approximation}

Given $n\in\N$ and $x^{0,n-1},y^{0,n-1}\in A^n$ the (non normalized) Hamming distance between $x$ and $y$ is 
\begin{equation}\label{def-Hamming-dist}
\bar{d}_n(x^{0,n-1},y^{0,n-1})=\sum_{i=0}^{n-1} \un_{\{x^i\neq y^i\}}\,.
\end{equation}
Now, given two shift-invariant probability measures $\mu,\nu$ on $\Omega$, denote by $\mu_n$ and $\nu_n$ their
projections on $A^n$, and define their $\bar{d}_n$-distance by
\[
\bar{d}_n(\mu_n,\nu_n)=\inf \sum_{x^{0,n-1}\in A^n}  \sum_{y^{0,n-1}\in A^n} \bar{d}_n(x,y)\,
\mathbb{P}_n(x^{0,n-1},y^{0,n-1})
\]
where the infimum is taken over all the joint shift-invariant probability distributions $\mathbb{P}_n$ on $A^n\times A^n$
such that $\sum_{y^{0,n-1}\in A^n} \mathbb{P}_n(x^{0,n-1},y^{0,n-1})=\mu_n(x^{0,n-1})$
and $\sum_{x^{0,n-1}\in A^n} \mathbb{P}_n(x^{0,n-1},y^{0,n-1})=\nu_n(y^{0,n-1})$.
By \cite[Theorem I.9.6, p. 92]{shieldsbook}, the limit following exists:
\begin{equation}\label{dbarmunu}
\bar{d}(\mu,\nu):=\lim_{n\to\infty} \frac{1}{n}\,\bar{d}_n(\mu_n,\nu_n)
\end{equation}
and defines a distance on the set of shift-invariant probability measures. It induces a finer topology than the weak topology and, in
particular, the $\bar{d}$-limit of ergodic measures is ergodic, and the entropy is $\bar{d}$-continuous on the class of ergodic measures.\footnote{These two properties are false in the weak topology.}

Next, given $n\in\N$ and a shift-invariant probability measure $\nu$ on $\Omega$, define the $n$-block relative entropy of
$\nu$ with respect to $\mu_\phi$ by
\[
H_n(\nu|\mu_\phi)=\sum_{x^{0,n-1}\in\,  A^n} \nu_n(x^{0,n-1}) \log\frac{\nu_n(x^{0,n-1})}{\mu_{\phi,n}(x^{0,n-1})}\,.
\]
One can easily prove that the following limit exists and defines the relative entropy of $\nu$ with respect to $\mu_\phi$:
\begin{equation}\label{relative-entropy}
\lim_{n\to\infty}\frac{1}{n}\, H_n(\nu_n|\mu_{\phi,n})=:h(\nu|\mu_\phi)=P(\phi)-\int \phi\, \dd \nu-h(\nu)
\end{equation}
where $P(\phi)$ is the topological pressure of $\phi$:
\[
P(\phi)=\lim_{n\to\infty} \frac{1}{n}\log \sum_{a^{0,n-1}\in A^n} \e^{\sup\big\{S_n\phi(x): x\in[a^{0,n-1}]\big\}}\,.
\]
This limit exists for any continuous $\phi$.
(To prove \eqref{relative-entropy}, we use that there exists a positive sequence $(\varepsilon_n)_n$ going to $0$ such that, for any $a^{0,n-1}\in A^n$ and any $x\in [a^{0,n-1}]$,  $\mu_\phi([a^{0,n-1}])/\exp(-nP(\phi)+S_n\phi(x))$ is bounded below by $\exp(-n\varepsilon_n)$ and above by $\exp(-n\varepsilon_n)$.)
By the variational principle, $h(\nu|\mu_\phi)\geq 0$ with equality if and only if $\nu=\mu_\phi$ (recall that $\mu_\phi$ is the unique equilibrium state of $\phi$). We refer to \cite{1975Walters} for details.
We can now formulate the first theorem of this section.
\begin{theorem}\label{theo-pinsker}
For every shift-invariant probability measure $\nu$ on $\Omega$ and for all $n\in\N$, we have
\begin{equation}\label{pinsker_n}
\bar{d}_n(\nu_n,\mu_{\phi,n})\leq \pcte{theo-pinsker}\,\sqrt{nH_n(\nu_n|\mu_{\phi,n})}
\end{equation}
where $\pcte{theo-pinsker}=\sqrt{2\cte{maintheo}}$. In particular
\begin{equation}\label{pinsker}
\bar{d}(\nu,\mu_\phi)\leq \pcte{theo-pinsker}\, \sqrt{h(\nu|\mu_\phi)}\,.
\end{equation}
\end{theorem}

\begin{proof}
For a function $f: A^n\to\R$, define for each $i=0,\ldots,n-1$
\[
\delta_i(f)=\sup\{|f(a^{0,n-1})-f(b^{0,n-1})|: a^j=b^j,\,\forall j\neq i\}\,.
\]
We obviously have that for all $a^{0,n-1},b^{0,n-1}\in A^n$
\[
|f(a^{0,n-1})-f(b^{0,n-1})|\leq \sum_{j=0}^{n-1} \un_{\{a^j\neq b^j\}} \delta_j(f)\,.
\]
A function $f: A^n\to\R$ such that $\delta_j(f)=1$, $i=0,\ldots,n-1$ is $1$-Lipschitz
for the Hamming distance \eqref{def-Hamming-dist}. We now consider the set of functions
\[
\mathcal{H}(n,\phi)=
\left\{f: A^n\to\R : f\;1\textup{-Lipschitz for}\;\bar{d}_n\;, \int_{A^n} f\, \dd \mu_{\phi,n}=0\right\}\,.
\]
We can identify a function $f\in \mathcal{H}(n,\phi)$ with a function $\tilde{f}:\Omega^n\to\R$ in a natural
way: $\tilde{f}(x_0,\ldots,x_{n-1})=f(\pi(x_0),\ldots,\pi(x_{n-1}))$ where $\pi:\Omega\to A$ is defined by 
$\pi(x)=x^0$. We obviously have $\int \tilde{f}\, \dd\mu_\phi=0$ and it is easy to check that $\Lip_j(\tilde{f})=\delta_j(f)=1$, $j=0,\ldots,n-1$. Therefore we can apply the Gaussian
concentration bound \eqref{gcb} to get
\begin{equation}\label{BG1}
\int_{A^n} \e^{\xi f} \dd\mu_{\phi,n}
\leq \e^{\cte{maintheo}n \xi^2}\,,\; \textup{for all}\;f\in \mathcal{H}(n,\phi)\,\textup{and for all}\,\xi\in\R\,.
\end{equation}
We now apply an abstract result \cite[Theorem 3.1]{BobkovGotze} which says that \eqref{BG1} is equivalent to
\[
\bar{d}(\nu_n,\mu_{\phi,n}) \leq \sqrt{2\cte{maintheo}n H_n(\nu_n|\mu_{\phi,n})}\quad\textup{for all probability measures}\; \nu_n\;\textup{on}\;A^n\,.
\]
Hence \eqref{pinsker_n} is proved. To get \eqref{pinsker}, divide by $n$ on both sides and take the limit $n\to\infty$ and
use \eqref{dbarmunu} and \eqref{relative-entropy}.
\end{proof}

We now give an application of inequality \eqref{pinsker}.  Let
\[
\phi_1(x)=\log\mu_\phi(x^0)\quad\textup{and}\quad\phi_n(x)=\log\mu_\phi(x^{n-1}|x^{0,n-2}), n\geq 2.
\]
The equilibrium state for $\phi_n$ is a $(n-1)$-step Markov measure.
One can prove that in the weak topology $(\mu_{\phi_n})_n$ converges to $\mu_\phi$, but one cannot get any speed of convergence.
We get the following upper bound on the speed of convergence of $(\mu_{\phi_n})_n$ to $\mu_\phi$ in the finer $\bar{d}$ topology.

\begin{corollary}
Assume, without loss of generality, that $\phi$ is normalized in the sense that
\[
\sum_{a\in A} \e^{\phi(ax)}=1,\;\forall x\in\Omega.
\]
Then there exists $n_\phi\geq 1$ such that, for all $n\geq n_\phi$, we have
\begin{equation}\label{dbarMarkovapprox}
\bar{d}(\mu_{\phi_n},\mu_\phi)\leq \rho_\phi\, \var_n(\phi)
\end{equation}
where
\[
\rho_\phi= \sqrt{2|A|\,\cte{maintheo}}\,(\e-1) \e^{\frac{3}{2}\|\phi\|_\infty}\,.
\]
\end{corollary}
More details on how to normalize a potential are given in Subsection \ref{sec:normalisation}.
\begin{proof}
Using \eqref{relative-entropy} and the variational principle we get
\begin{equation}
\label{mezcal}
h(\mu_{\phi_n}|\mu_{\phi})
=-\int \phi\,\dd\mu_{\phi_n}-h(\mu_{\phi_n})
= \int(\phi_n-\phi)\, \dd\mu_{\phi_n}.
\end{equation}
Indeed, since $\phi$ and $\phi_n$ are normalized, we have in particular that $P(\phi)=P(\phi_n)=0$, and by the variational principle
$h(\mu_{\phi_n})=-\int \phi_n \dd\mu_{\phi_n}$.
Now
\begin{align}
\nonumber
\int(\phi_n-\phi)\, \dd\mu_{\phi_n}
&= \int \log \left(\frac{\e^{\phi_n}}{\e^{\phi}}\right)\, \dd\mu_{\phi_n}
= \int \log\left(1+ \frac{\e^{\phi_n}-\e^{\phi}}{\e^{\phi}}\right) \dd\mu_{\phi_n}\\
\label{tequila}
& \leq \int \frac{\e^{\phi_n}-\e^{\phi}}{\e^{\phi}}\, \dd\mu_{\phi_n}
\end{align}
where we used the inequality $\log (1+u)\leq u$ for all $u>-1$.  Now using the shift-invariance of $\mu_{\phi_n}$ and
replacing $\e^{\phi_n}$ by $\e^{\phi_n}-\e^{\phi}+\e^{\phi}$ we get
\begin{align}
\nonumber
&\int \frac{\e^{\phi_n}-\e^{\phi}}{\e^{\phi}}\, \dd\mu_{\phi_n}
= \int \dd\mu_{\phi_n}(x) \sum_{a\in A} \e^{\phi_n(ax)}\,  \frac{\e^{\phi_n(ax)-\e^{\phi(ax)}}}{\e^{\phi(ax)}}\\
\nonumber
&= \int \dd\mu_{\phi_n}(x) \sum_{a\in A} \frac{\big(\e^{\phi_n(ax)}-\e^{\phi(ax)}\big)^2}{\e^{\phi(ax)}}
+\int \dd\mu_{\phi_n}(x) \sum_{a\in A} \big(\e^{\phi_n(ax)}-\e^{\phi(ax)}\big)\\
\label{cocktail}
& \leq |A|\, \e^{\|\phi\|_\infty}\, (\|\e^{\phi_n}-\e^{\phi}\|_\infty)^2
\end{align}
where we used that $ \sum_{a\in A} (\e^{\phi_n(ax)}-\e^{\phi(ax)})=0$. Combining \eqref{pinsker}, \eqref{mezcal}, \eqref{tequila}
and \eqref{cocktail} we thus obtain 
\[
\bar{d}(\mu_{\phi_n},\mu_\phi)\leq \sqrt{2|A|\,\cte{maintheo}\e^{\|\phi\|_\infty}}\, \|\e^{\phi_n}-\e^{\phi}\|_\infty.
\]
It remains to estimate $\|\e^{\phi_n}-\e^{\phi}\|_\infty$ in terms of $\var_n(\phi)$. We have
\[
\|\e^{\phi_n}-\e^{\phi}\|_\infty=\big\|\e^{\phi_n}-\e^{\phi} \big\|_\infty\leq
\e^{\|\phi\|_\infty} \big\|\e^{\phi-\phi_n}-1 \big\|_\infty \leq
(\e-1)\e^{\|\phi\|_\infty}\|\phi-\phi_n\|_\infty
\]
provided that $\|\phi-\phi_n\|_\infty<1$, where we used the inequality $|\e^u-1|\leq (\e-1)|u|$ valid for $|u|<1$.
Finally, since $\|\phi-\phi_n\|_\infty\leq \var_n(\phi)$, we define $n_\phi$ to be the smallest integer sucht
$\var_n(\phi)<1$ and we can take 
\[
\rho_\phi= \sqrt{2|A|\,\cte{maintheo}}\, (\e-1)\e^{\frac{3}{2}\|\phi\|_\infty}.
\] 
We thus proved \eqref{dbarMarkovapprox}.
\end{proof}
Let us mention the paper \cite{2002FernandezGalves} in which the authors obtain the same bound for the speed of 
convergence of Markov approximation, up to the constant. Their approach is a direct estimation of $\bar{d}(\mu_{\phi_n},\mu_\phi)$ 
by using a coupling method.
The point here is to obtain the same speed of convergence as an easy corollary of inequality \eqref{pinsker}. 
Let us remark that from \eqref{pouic} we get a worse result since we end up with a bound proportional to $\sqrt{\var_n(\phi)}$. The 
trick which leads to the correct bound was told us by Daniel Takahashi.

\subsubsection{Shadowing of orbits}

Let $A$ be a Borel subset of  $\Omega$ such that $\mu_\phi(A)>0$ and define for all $n\in\N$
\begin{equation*}
\mathcal{S}_A(x,n)=\frac{1}{n} \inf_{y\in A} \sum_{j=0}^{n-1} d_\theta(T^j x,T^j y)
\end{equation*}
A basic example of set $A$ is a cylinder set $[a^{0,k-1}]$. The quantity $\mathcal{S}_A(x,n)$, which lies between $0$ and $1$,
measures how we can trace, in the best possible way, the orbit of some initial condition not in $A$
by an orbit starting in $A$. 

\begin{theorem}
For any Borel subset $A\subset \Omega$ such that $\mu_\phi(A)>0$, for any $n\in\N$ and for any $u>0$
\[
\mu_\phi\left\{x \in \Omega: \mathcal{S}_A(x,n)\geq \frac{u_A+u}{\sqrt{n}}\right\}
\leq \e^{-\frac{u^2}{4\cte{maintheo}}}
\]
where 
\[
u_A=2\sqrt{-\, \cte{maintheo}\ln\mu_\phi(A)}\,.
\]
\end{theorem}
%For instance, take $A$ to be a cylinder set with $\mu_\phi(A)=1/2$.

We give a shorter and simpler proof than in \cite{colletmartinezschmitt}.
\begin{proof}
Let $K(x_0,\ldots,x_{n-1})=\frac{1}{n} \inf_{y\in A} \sum_{j=0}^{n-1} d_\theta(x_{j},T^j y)$. One can easily check that 
\[
\Lip_{\theta,i}(K)= \frac{1}{n}, \;\forall i=0,\ldots,n-1.
\]
It follows from \eqref{dev-gcb} that
\[
\mu_\phi\left\{\mathcal{S}_A(x,n)\geq \int \mathcal{S}_A(y,n)\, \dd\mu_\phi(y)+\frac{u}{\sqrt{n}}\right\}\leq \e^{-\frac{u^2}{4\cte{maintheo}}}
\]
for all $n\geq 1$ and for all $u>0$. We now need an upper bound for $\int \mathcal{S}_A(y,n)\dd\mu_\phi(y)$.
We simply observe that by \eqref{gcb} and the definition of $ \mathcal{S}_A(\cdot,n)$
\begin{align*}
\mu_\phi(A) &=\int \e^{-\xi \mathcal{S}_A(x,n)}\un_A(x)\, \dd\mu_\phi(x)
\leq \int \e^{-\xi \mathcal{S}_A(x,n)}\, \dd\mu_\phi(x)\\
& \leq \e^{-\xi \int \mathcal{S}_A(y,n)\, \dd\mu_\phi(y)} \e^{\frac{\cte{maintheo}\xi^2}{n}}
\end{align*}
for all $\xi>0$. Hence
\[
\int \mathcal{S}_A(y,n)\, \dd\mu_\phi(y) \leq \frac{\cte{maintheo}\xi}{n}+\frac{1}{\xi}\ln(\mu_\phi(A)^{-1})
\]
Optimizing this bound over $\xi>0$ gives
\[
\int \mathcal{S}_A(y,n)\, \dd\mu_\phi(y) \leq 2\sqrt{\frac{\cte{maintheo}\ln(\mu_\phi(A)^{-1})}{n}}.
\]
The theorem follows at once.
\end{proof}

\subsubsection{Almost-sure central limit theorem}

It was proved in \cite[Chapter 2]{maume-phd} that $(\Omega, T,\mu_\phi)$ satisfies the central limit theorem for the class of $d_\theta$-Lipschitz functions $f:\Omega\to\R$ such $\int f\dd\mu_\phi=0$, that is, for any such $f$ the process $\{f\circ T^n\}_{n\geq 0}$ satisfies
\begin{equation}\label{clt}
\mu_\phi\Big(x:\frac{S_k f(x)}{\sqrt{k}}\leq u\Big)=
\int \un_{\big\{\frac{S_k f(x)}{\sqrt{k}}\leq t\big\}}\dd\mu_\phi(x) \xrightarrow[]{k\to\infty} G_{0,\sigma^2(f)}((-\infty,t])
\end{equation}
where 
\[
\sigma^2(f)=\int f^2\, \dd\mu_\phi + 2\sum_{i\geq 1} \int f\cdot f\circ T^i \, \dd\mu_\phi\in [0,+\infty)\,.
\]
If $\sigma^2(f)> 0$, $G_{0,\sigma^2}$ denotes the law of a Gaussian random variable with mean $0$ and
variance $\sigma^2(f)$, that is,
\[
\dd G_{0,\sigma^2(f)}(u)=\frac{1}{\sigma\sqrt{2\pi}} \, \e^{-\frac{u^2}{2\sigma^2(f)}}\dd u, \; u\in\R\,.
\] 
When $\sigma^2(f)=0$ we set $G_{0,0}=\delta_0$, the Dirac mass at zero. 
\begin{remark}
In fact, a more general statement was proved in \cite[Chapter I.2]{maume-phd}. Namely, \eqref{clt} holds when
$\phi$ is such that $\sum_k \epsilon_k<+\infty$ and $f\in\mathcal{L}_\phi$.
\end{remark}

Now, for each $N\geq 1$ and $x\in\Omega$, define the probability measure
\begin{equation}\label{def-AN}
\A_N(x)=\frac{1}{L_N} \sum_{n=1}^N \frac{1}{n}\, \gdelta_{\frac{S_n f(x)}{\sqrt{n}}}
\end{equation}
where $L_N=\sum_{n=1}^N \frac{1}{n}$ and where, as usual, $\delta_u$ is the Dirac mass at point $u\in\R$. 
Of course, $L_N=\log N + \Oun$. Notice that $\A_N$ is a random probability measure.
Finally, the Wasserstein distance between two probability measures 
$\nu$, $\nu'$ on the Borel sigma-algbra $\mathscr{B}(\R)$ is
\[
W_1(\nu,\nu')=\inf_{\pi\in \Pi(\nu,\nu')}\int d_\theta(x,x')\, \dd\pi(x,x')
\]
where the infimum is taken over all probability measures such that
\[
\int \pi(B,x')\, \dd x'=\nu(B)\quad \text{and}\quad \int \pi(x,B)\, \dd x=\nu'(B)
\] 
for any Borel subset of $\R$.
By the Kantorovich-Rubinstein duality theorem, $W_1(\nu,\nu')$ is equal to the Kantorovich distance which is the supremum
of $\int \ell \, \dd\nu -\int \ell \, \dd\nu'$ over the set of $1$-Lipschitz functions $\ell:\R\to\R$. We refer to \cite{dudley} for background and proofs.

Now we can formulate the almost-sure central limit theorem.

\begin{theorem}\label{asclt}
Let $f:\Omega\to\R$ be a $d_\theta$-Lipschitz function. Then,
for $\mu_\phi$ almost every $x\in\Omega$, we have
\[
W_1(\A_N(x) ,G_{0,\sigma^2(f)})\xrightarrow[N\to+\infty]{} 0\,.
\]
\end{theorem}
We make several comments.
Recall that the Wasserstein distance metrizes the weak topology on the set of probability measures $\nu$ on $\mathscr{B}(\R)$.
Moreover, if $(\nu_n)_{n\geq 1}$ is a sequence of probability measures on $\mathscr{B}(\R)$ and $\nu$ a probability
measure on $\mathscr{B}(\R)$, then
\[
\lim_{n\to\infty}W_1(\nu_n,\nu)=0
\quad \Longleftrightarrow\quad
\nu_n\lawto \nu\quad\text{and}\;\int |u|\,  \dd\nu_n(u)\xrightarrow[]{n\to+\infty} \int |u|\,\dd\nu(u)
\]
where ``$\lawto$'' means weak convergence of probability measures on $\mathscr{B}(\R)$. 
\newline
To compare with \eqref{clt}, observe that Theorem \ref{asclt} implies that for $\mu_\phi$-almost every $x$, $\A_N(x) \lawto G_{0,\sigma^2(f)}$,
which in turn implies that
\[
\int \un_{\{u\leq t\}} \dd\A_N(u)= \frac{1}{L_N} \sum_{n=1}^N \frac{1}{n} \un_{\{S_nf/\sqrt{n}\leq t\}} 
\xrightarrow[N\to+\infty]{} G_{0,\sigma^2(f)}((-\infty,t]).
\]
Therefore, the expectation with respect to $\mu_\phi$ in \eqref{clt} is replaced by a pathwise logarithmic average in the almost-sure central limit theorem.

\begin{proof}
The proof follows from an abstract theorem proved in \cite{2005ChazottesColletSchmitt}. In words, that theorem says the following.
Let $(X_n)_{n\geq 0}$ be a stochastic stationary process where the $X_n$'s are random variables taking values in $\Omega$. Assume that if $f:\Omega\to\R$ is $d$-Lipschitz and such that $\E[f(X_0)]=0$, then it satisfies the central limit theorem, that is, for all $u\in\R$,
\[
\mathds{P}\left( \frac{\sum_{j=0}^{n-1} f(X_j)}{\sqrt{n}}\leq u\right)\xrightarrow[]{n\to\infty}
G_{0,\sigma^2(f)}((-\infty,u])
\]
where $\sigma^2(f):=\E[f^2(X_0)]+2\sum_{\ell \geq 1} \E[f(X_0)f(X_\ell)]$ is assumed to be $\neq 0$. 
Moreover, assume that the process $(X_n)_{n\geq 0}$ satisfies the following variance inequality: There exists $C>0$ such that for all separately $d$-Lipschitz functions $K:\Omega^n\to\R$ for some distance $d$ on $\Omega$, 
\[
\E\big[(K(X_0,\ldots,X_{n-1})-\E[K(X_0,\ldots,X_{n-1})])^2\big]\leq C\sum_{i=0}^{n-1} \Lip_i(K)^2\,. 
\]
Then, the conclusion is that, almost surely,  
\[
\frac{1}{L_N}\sum_{n=1}^N \frac{1}{n}\delta_{\frac{X_0+\cdots+X_{n-1}}{\sqrt{n}}}
\]
converges in Wasserstein distance (or, equivalently, in Kantorovich distance) to $G_{0,\sigma^2(f)}((-\infty,u])$. We apply this abstract theorem to the process $(x,Tx,\ldots)$ where $x\in\Omega$ is distributed according to $\mu_\phi$ with $\Omega=A^\N$
and $d=d_\theta$.
Since we have \eqref{clt} and \eqref{devroye}, the theorem follows.
\end{proof}

\begin{remark}
The previous result relies only upon the variance inequality \eqref{devroye}, which is much weaker than the Gaussian
concentration bound of Theorem \ref{maintheo}. On the one hand, the variance inequality \eqref{devroye} should be true
for less regular potentials than the ones we consider here. On the other hand, the Gaussian concentration bound should provide a strengthening of Theorem \ref{asclt}, namely a speed of convergence.
\end{remark}

%%%% PROOF OF MAIN THEOREM
\section{Proof of Theorem \ref{maintheo}}\label{sec:proof-main-thm}

We follow the proof given in \cite{2012ChazottesGouezel} with the appropriate modifications to go beyond Lipschitz potentials.

\subsection{Some preparatory results}\label{sec:normalisation}

It is convenient to normalize the potential $\phi$ or, equivalently, the operator $P_\phi$ in the following way. 
We use the notations of Theorem \ref{backgroundthm}.
Let
\[
\widetilde{P}_\phi f=\lambda_\phi^{-1} h_\phi^{-1} P_\phi(f h_\phi).
\]
Thus
\begin{equation}\label{Pphinormalise}
\widetilde{P}_\phi 1= 1\quad\text{and}\quad \widetilde{P}^*_\phi \mu_\phi=\mu_\phi.
\end{equation}
Let $g$ denote the inverse of the Jacobian of $T$, and $g^{(k)}$ the inverse of the Jacobian of $T^k$, that is, 
\[
g=\frac{h_\phi}{\lambda_\phi h_\phi\circ T}\exp(\phi)\quad\text{and}\quad g^{(k)}=\frac{h_\phi}{\lambda_\phi^k h_\phi\circ T}\exp\Big(\sum_{i=0}^{k-1}\phi\circ T^i\Big)\,.
\]
(Of course $g=g^{(1)}$.) Therefore we have
\begin{equation}\label{itererPphi}
\widetilde{P}_\phi f(x)=\sum_{Ty=x} g(y)f(y)\quad\text{and}\quad \widetilde{P}_\phi^k f(x)=\sum_{T^ky=x} g^{(k)}(y)f(y)\,.
\end{equation}
Estimate \eqref{ineq-vero} now takes the form
\begin{equation}\label{ineq-vero-normalisee}
\left\| \widetilde{P}^n_\phi f-\int f\dd\mu_\phi\right\|_\infty \leq 
\cte{backgroundthm}\,\|f\|_{\mathcal{L}_\phi}\, \epsilon_n,\; n\geq 1,
\end{equation}
for any $f\in\mathcal{L}_\phi$.
Finally, we will need the following distortion estimate. Let $x,y\in\Omega$ such that $x^i=y^i$ for $i=0,\dotsc,n-1$ and $x',y'\in\Omega$ such that
$T^k x'=x$ and $T^ky'=y$. Then it is easy to check (see \cite[Chapter 2]{maume-phd}) that, for any $k$,
\begin{equation}\label{distortion}
\left| 1-\frac{g^{(k)}(x')}{g^{(k)}(y')}\right|\leq \pcte{distortion}\, d_\phi(x,y)
\end{equation}
for some constant $\pcte{distortion}>0$ depending only on $\phi$.

We will use the following inequality relating the distances $d_\theta$ and $d_\phi$.
\begin{lemma}\label{comparaisondistances}
Suppose that $W_n(\phi)=O(\theta^n)$, $n\geq 1$, or 
\begin{equation}\label{W/W=1}
\lim_n \frac{W_n(\phi)}{W_{n+1}(\phi)}=1\,.
\end{equation}
Then there exists $\pcte{comparaisondistances}>0$
\[
\sup_n \frac{\theta^n}{W_n(\phi)}\leq \pcte{comparaisondistances}
\]
or, equivalently,
\[
d_\theta(x,y)\leq \pcte{comparaisondistances} d_\phi(x,y)
\]
for all $x,y$.
\end{lemma}

\begin{proof}
The statement is trivial when $W_n(\phi)=O(\theta^n)$. If \eqref{W/W=1} holds, then there exists $n_0$ such that for all $n\geq n_0$
\[
\frac{W_n(\phi)}{W_{n+1}(\phi)}\leq \frac{1}{\theta},
\]
hence $W_n(\phi)\geq \theta^{n-n_0} W_{n_0}(\phi)$. Then the desired inequalities follow easily from
the definitions.
\end{proof}

\subsection{Proof of Theorem \ref{maintheo}}

Fix a separately $d_\theta$-Lipschitz function $K:\Omega^n\to\R$. It is convenient to think of it as a function on $\Omega^\N$ depending only on the first $n$
coordinates, therefore $\Lip_{\theta,i}(K)=0$ for $i\geq n$.
We endow $\Omega^{\N}$ with the measure $\mu^\infty$ obtained as the limit when $k\to\infty$
of the measure $\mu^\infty_k$ on $\Omega^k$ given by $\dd\mu^\infty_k(x_0,\dotsc,x_{k-1})=\dd\mu_\phi(x_0) \delta_{x_1=Tx_0}\dotsm\delta_{x_{k-1}=Tx_{k-2}}$. 
On $\Omega^\N$, let $\mathcal{F}_p$ be the $\sigma$-algebra of
events depending only on the coordinates $(x_j)_{j\geq p}$ (this is a
decreasing sequence of $\sigma$-fields). 
We want to write the function $K$ as a sum of reverse martingale differences with respect
to this sequence. Therefore, let $K_p = \E(K|\mathcal{F}_p)$ and
$D_p=K_{p}-K_{p+1}$. More precisely,
\begin{align*}
K_p(x_p, x_{p+1},\dotsc)
& =\E(K|\mathcal{F}_p) (x_p,x_{p+1},\dotsc)\\
& = \E(K(X_0,\dotsc, X_{p-1}, x_p,\dotsc) |X_p=x_p)
\\ & = \sum_{T^p(y)=x_p} g^{(p)}(y)K(y,\dotsc, T^{p-1}y, x_p,\dotsc).
\end{align*}
The function $D_p$ is $\mathcal{F}_p$-measurable and
$\E(D_p|\mathcal{F}_{p+1}) =0$. Moreover
\begin{equation}\label{sumDp}
K-\E(K) = \sum_{p\geq 0} D_p.
\end{equation}
We then apply Azuma-Hoeffding inequality (see e.g.~\cite[Page 68]{ledoux}) which says that
\begin{equation}\label{AHineq}
\E\left(\e^{\sum_{p=0}^{P-1} D_p}\right) \leq \e^{\frac12 \sum_{p=0}^{P-1} \|D_p\|_\infty^2}.
\end{equation}
Therefore, the point is to obtain a good bound on $D_p$. This is the claim of the following lemma.
\begin{lemma}\label{mainlemma}
There exists $\cte{mainlemma}>0$, depending only on $\phi$, such that for any $p\in\N$ one has
\[
\|D_p \|_\infty \leq \cte{mainlemma} \sum_{i=0}^{p}\epsilon_{p-i}\sum_{j=0}^i \Lip_{\theta,j}(K)\,\theta^{i-j}+ \Lip_{\theta,p}(K)\,.
\]
\end{lemma} 
Using this lemma and applying Young's inequality for convolutions \cite[p. 316]{bullen} twice we obtain 
\begin{align*}
\MoveEqLeft \sum_{p=0}^{P-1} \|D_p\|_\infty^2 \\
& \leq 2\cte{mainlemma}^2\sum_{p=0}^{P-1}  \left(\sum_{i=0}^{p}\epsilon_{p-i} \sum_{j=0}^i 
\Lip_{\theta,j}(K)\,\theta^{i-j}\right)^2+ 2\sum_{p=0}^{P-1} \Lip_{\theta,p}(K)^2 \\
& \leq 2\cte{mainlemma}^2 \left( \sum_{k\geq  1}\epsilon_k\right)^2 \;\sum_{p=0}^{P-1}\left( \sum_{j=0}^p \Lip_{\theta,j}(K)\,\theta^{p-j}\right)^2+ 2\sum_{p=0}^{P-1} \Lip_{\theta,p}(K)^2 \\
& \leq 2\left(\cte{mainlemma}^2  (1-\theta)^{-2}\left( \sum_{k\geq 1} \epsilon_k\right)^2+1\right) \;\sum_{p=0}^{P}\Lip_{\theta,p}(K)^2.
\end{align*}
\begin{remark}
If $u=(u_n)_n$ and $v=(v_n)_n$ are sequences of reals, their convolution $u\star v$ is given by $(u\star v)_n=\sum_{k=0}^n u_k v_{n-k}$.
Young's inequality tells us that if $u\in\ell^p(\N)$, $u\in\ell^q(\N)$ and $1\leq p,q,r \leq \infty$ with $r^{-1}+1=p^{-1}+q^{-1}$, then
\[
\|u\star v\|_r \leq \|u\|_p \|v\|_q\,.
\] 
We used it twice with $r=2$, $p=2$ and $q=1$.
\end{remark}
Notice that by assumption and by Theorem \ref{backgroundthm} we have $\sum_{k\geq 1} \epsilon_k<+\infty$.
Therefore, using  \eqref{AHineq} at a fixed index $P$ and then letting $P$ tend to infinity, we get by the dominated convergence
theorem
\[
\E\left(\e^{\sum_{p\geq 0} D_p}\right) \leq \e^{\frac12 \sum_{p\geq 0}\|D_p\|_\infty^2}
\]
which is, in view of \eqref{sumDp}, exactly \eqref{gcb} with
\[
\cte{maintheo}=1+\cte{mainlemma}^2 (1-\theta)^{-2} \left( \sum_{k\geq 1} \epsilon_k\right)^2\,.
\]
Now we are going to prove Lemma \ref{mainlemma} by proving that $K_p$ is close to an integral quantity. This is the content of the following lemma which is the core of the proof.
\begin{lemma}\label{lem_controle_Kp}
There exists $\cte{lem_controle_Kp}>0$, depending only on $\phi$, such that, for all $p\in\N$, 
\begin{align*}
\MoveEqLeft[15] \left|K_p(x_p,\dotsc) -\int K(y,\dotsc, T^{p-1}y, x_p,\dotsc)\,\dd \mu_\phi(y)\right|\\
& \leq \cte{lem_controle_Kp} \sum_{i=0}^{p-1} \epsilon_{p-i}\, 
\sum_{j=0}^i \Lip_{\theta,j}(K)\, \theta^{i-j}
\end{align*}
where
\[
\cte{lem_controle_Kp}=\cte{backgroundthm}(\pcte{distortion}+2\pcte{comparaisondistances}).
\]
\end{lemma}

\begin{proof}[Proof of Lemma~\ref{mainlemma}]
Applying Lemma \ref{lem_controle_Kp} yields
\begin{align*}
\MoveEqLeft[10] |K_p(x_p,x_{p+1},\dotsc)-K_p(x'_p,x_{p+1},\dotsc)|\\
& \leq
2 \cte{lem_controle_Kp} \sum_{i=0}^{p-1} \epsilon_{p-i}\, 
\sum_{j=0}^i \Lip_{\theta,j}(K)\, \theta^{i-j} + \Lip_{\theta,p}(K)\,.
\end{align*}
Averaging $K_p(x'_p,x_{p+1},\dotsc)$ over the preimages of $x'_p$ we get exactly $K_{p+1}(x_{p+1},\dotsc)$, hence the previous bound holds for $|D_p|$, proving the lemma.
\end{proof}

\begin{proof}[Proof of Lemma~\ref{lem_controle_Kp}]
Let us fix a point $x_*$ in $\Omega$ and decompose $K_p$ as
\begin{align*}
K_p(x_p,\dotsc)&=\sum_{i=0}^{p-1} \sum_{T^p(y)=x_p} g^{(p)}(y)(K(y,\dotsc,T^{i-1} y,T^i y, x_*,\dotsc,x_*, x_p,\dotsc)
\\
&\hphantom{=\sum_{i=1}^{p-1} \sum_{T^p(y)=x_p} g^{(p)}(y)(} - K(y,\dotsc, T^{i-1}y, x_*,\dotsc, x_*, x_p,\dotsc))\\ &
 \ \ +K(x_*,\dotsc,x_*, x_p,\dotsc).
\end{align*}
For fixed $i$, we can group together those points $y\in T^{-p}(x_p)$ which have the same image under $T^i$, splitting the sum
$\sum_{T^p(y)=x_p}$ as $\sum_{T^{p-i}(z)=x_p} \sum_{T^i(y)=z}$. Since the jacobian is multiplicative,
one has $g^{(p)}(y) = g^{(i)}(y)g^{(p-i)}(z)$. Let us define two functions $f_i$ and $H$ as follows:
\begin{equation*}
%\label{eq_define_fi}
\begin{split}
f_i(z)&=\sum_{T^i y = z}g^{(i)}(y)(K(y,\dotsc, T^{i-1} y,T^i y, x_*,\dotsc,x_*, x_p,\dotsc)
\\ & \hphantom{= \sum_{T^i y = z}g^{(i)}(y)(}
- K(y,\dotsc, T^{i-1}y, x_*,\dotsc, x_*, x_p,\dotsc))
\\ & = \sum_{T^i y = z}g^{(i)}(y) H(y,\dotsc, T^i y).
\end{split}
 \end{equation*}
Bearing in mind \eqref{itererPphi}, we obtain
\begin{equation*}
K_p(x_p,\dotsc)=\sum_{i=0}^{p-1} \widetilde{P}_\phi^{p-i} f_i(x_p) + K(x_*,\dotsc,x_*, x_p,\dotsc).
\end{equation*}
Now we want to prove that $f_i\in\mathcal{L}_\phi$ to use \eqref{ineq-vero-normalisee}.
First observe that for any $z\in\Omega$
\[
|f_i(z)|\leq \sum_{T^i y = z}g^{(i)}(y)\, \Lip_{\theta,i}(K)\, d_\theta(x_*,T^iy)\leq \Lip_{\theta,i}(K)
\]
since $d_\theta(x_*,T^iy)\leq 1$ and $\sum_{T^i y = z}g^{(i)}(y)=1$. Hence 
\[
\|f_i\|_\infty\leq \Lip_{\theta,i}(K)\,.
\]
We now estimate the $d_\phi$-Lipschitz norm of $f_i$. We write
\begin{equation}
\label{pouic}
\begin{split}
f_i(z)-f_i(z')= {} &\sum (g^{(i)}(y)-g^{(i)}(y')) H(y,\dotsc, T^i y)
\\& +
\sum g^{(i)}(y')(H(y,\dotsc, T^i y)-H(y',\dotsc, T^i y'))
\end{split}
\end{equation}
where $z$ and $z'$ are two points in the same partition element, and
their respective preimages $y$, $y'$ are paired according to the
cylinder of length $i$ they belong to. 
Using the distorsion control \eqref{distortion} we have
\[
|g^{(i)}(y)-g^{(i)}(y')| \leq \pcte{distortion}\, g^{(i)}(y)\, d_\phi(z,z')
\]
hence the first sum in \eqref{pouic} is bounded in absolute value by
\[
\pcte{distortion}\, \Lip_{\theta,i}(K)\, d_\phi(z,z')\,. 
\]
For the second sum,
substituting successively each $T^j y$ with $T^j y'$, we have
\begin{align*}
|H(y,\dotsc, T^i y)-H(y',\dotsc, T^i y')|
& \leq 2\sum_{j=0}^i \Lip_{\theta,j}(K)\, d_\theta(T^j y, T^j y') \\
& \leq 2\sum_{j=0}^i \Lip_{\theta,j}(K)\, \theta^{i-j} d_\theta(z,z')\\
& \leq 2\pcte{comparaisondistances} \sum_{j=0}^i \Lip_{\theta,j}(K)\, \theta^{i-j} d_\phi(z,z')
\end{align*}
where we used Lemma \ref{comparaisondistances} for the third inequality. \newline
Summing over the different preimages of $z$, we deduce that
\begin{equation*}
%\label{lipphifi}
\|f_i\|_{\mathcal{L}_\phi} \leq (\pcte{distortion}+2\pcte{comparaisondistances})\sum_{j=0}^i \Lip_{\theta,j}(K)\, \theta^{i-j}.
\end{equation*}
Therefore we can apply \eqref{ineq-vero-normalisee} to get
\[
\left\| \widetilde{P}^{p-i}_\phi f_i - \int f_i\,\dd\mu_\phi\right\|_\infty \leq \cte{backgroundthm}\big(\pcte{distortion}+2\pcte{comparaisondistances}\big) \, 
\epsilon_{p-i}\, \sum_{j=0}^i \Lip_{\theta,j}(K)\, \theta^{i-j}\,.
\]
Summing those bounds, one obtains
\begin{align*}
\Bigl|K_p(x_p,\dotsc) &- \sum_{i=0}^{p-1} \int f_i\, \dd\mu_\phi - K(x_*,\dotsc,x_*, x_p,\dotsc)\Bigr|
\\
&\leq \cte{backgroundthm}\big(\pcte{distortion}+2\pcte{comparaisondistances}\big) \, \sum_{i=0}^{p-1} \epsilon_{p-i}\, 
\sum_{j=0}^i \Lip_{\theta,j}(K)\, \theta^{i-j}\,.
\end{align*}

Finally, when one computes the sum of the integrals of $f_i$, there
are again cancelations, leaving only $\int K(y,\dotsc, T^{p-1}y,x_p,\dotsc)\,\dd \mu_\phi(y)$.
\end{proof}

%%%%%%%%%%%%%%%%%%%%%%%%%%%%%%%%%%%%%%%%%%%

\bibliographystyle{abbrv}

\begin{thebibliography}{10}

\bibitem{BobkovGotze}
S.~G. Bobkov and F.~G\"{o}tze.
\newblock Exponential integrability and transportation cost related to
  logarithmic {S}obolev inequalities.
\newblock {\em J. Funct. Anal.}, 163(1):1--28, 1999.

\bibitem{boucheron_lugosi_massard}
S.~Boucheron, G.~Lugosi, and P.~Massart.
\newblock {\em Concentration inequalities}.
\newblock Oxford University Press, Oxford, 2013.
\newblock A nonasymptotic theory of independence, With a foreword by Michel
  Ledoux.

%\bibitem{BressaudFernandezGalves}
%X.~Bressaud, R.~Fern\'{a}ndez, and A.~Galves.
%\newblock Speed of {$\overline d$}-convergence for {M}arkov approximations of
%  chains with complete connections. {A} coupling approach.
%\newblock {\em Stochastic Process. Appl.}, 83(1):127--138, 1999.

\bibitem{bullen}
P.~Bullen.
\newblock {\em Dictionary of inequalities}.
\newblock Monographs and Research Notes in Mathematics. CRC Press, Boca Raton,
  FL, second edition, 2015.

\bibitem{chazottes_survey}
J.-R. Chazottes.
\newblock Fluctuations of observables in dynamical systems: from limit theorems
  to concentration inequalities.
\newblock In {\em Nonlinear dynamics new directions}, volume~11 of {\em
  Nonlinear Syst. Complex.}, pages 47--85. Springer, Cham, 2015.

\bibitem{chazottes-collet-redig2017}
J.-R. Chazottes, P.~Collet, and F.~Redig.
\newblock On concentration inequalities and their applications for {G}ibbs
  measures in lattice systems.
\newblock {\em Journal of Statistical Physics}, 169(3):504--546, Nov 2017.

\bibitem{2005ChazottesColletSchmitt}
J.-R. Chazottes, P.~Collet, and B.~Schmitt.
\newblock Statistical consequences of the {D}evroye inequality for processes.
  {A}pplications to a class of non-uniformly hyperbolic dynamical systems.
\newblock {\em Nonlinearity}, 18(5):2341--2364, 2005.

\bibitem{2012ChazottesGouezel}
J.-R. Chazottes and S.~Gou\"ezel.
\newblock Optimal concentration inequalities for dynamical systems.
\newblock {\em Comm. Math. Phys.}, 316(3):843--889, 2012.

\bibitem{2011ChazottesMaldonado}
J.-R. Chazottes and C.~Maldonado.
\newblock Concentration bounds for entropy estimation of one-dimensional
  {G}ibbs measures.
\newblock {\em Nonlinearity}, 24(8):2371--2381, 2011.

\bibitem{chazottes-ugalde-2005}
J.-R. Chazottes and E.~Ugalde.
\newblock Entropy estimation and fluctuations of hitting and recurrence times
  for {G}ibbsian sources.
\newblock {\em Discrete Contin. Dyn. Syst. Ser. B}, 5(3):565--586, 2005.

\bibitem{2012ColletGalves}
P.~Collet and A.~Galves.
\newblock Chains of infinite order, chains with memory of variable length, and
  maps of the interval.
\newblock {\em J. Stat. Phys.}, 149(1):73--85, 2012.

\bibitem{colletmartinezschmitt}
P.~Collet, S.~Mart\'{\i}nez, and B.~Schmitt.
\newblock Exponential inequalities for dynamical measures of expanding maps of
  the interval.
\newblock {\em Probab. Theory Related Fields}, 123(3):301--322, 2002.

\bibitem{dudley}
R.~Dudley.
\newblock {\em Real Analysis and Probability}.
\newblock CRC Press, 2018.


\bibitem{2002FernandezGalves} R.~Fern\'andez and A.~Galves. 
\newblock Markov approximations of chains of infinite order.
\newblock {\em Bull. Braz. Math. Soc., New Series}, 33(3): 295--306, 2002. 

\bibitem{2014GalloTakahashi}
S.~Gallo and D.~Y. Takahashi.
\newblock Attractive regular stochastic chains: perfect simulation and phase
  transition.
\newblock {\em Ergodic Theory Dynam. Systems}, 34(5):1567--1586, 2014.

\bibitem{1998Keller}
G.~Keller.
\newblock {\em Equilibrium states in ergodic theory}, volume~42 of {\em London
  Mathematical Society Student Texts}.
\newblock Cambridge University Press, Cambridge, 1998.

\bibitem{1997MaumeKondahSchmitt}
A.~Kondah, V.~Maume, and B.~Schmitt.
\newblock Vitesse de convergence vers l'\'etat d'\'equilibre pour des
  dynamiques markoviennes non h\"old\'eriennes.
\newblock {\em Ann. Inst. H. Poincar\'e Probab. Statist.}, 33(6):675--695,
  1997.

\bibitem{ledoux}
M.~Ledoux.
\newblock {\em The concentration of measure phenomenon}, volume~89 of {\em
  Mathematical Surveys and Monographs}.
\newblock American Mathematical Society, Providence, RI, 2001.

\bibitem{maume-phd}
V.~Maume-Deschamps.
\newblock {\em Propri\'et\'es de m\'elange pour des syst\`emes dynamiques
  markoviens}.
\newblock PhD thesis, Universit\'e de Bourgogne, 1998.

\bibitem{pollicott2000}
M.~Pollicott.
\newblock Rates of mixing for potentials of summable variation.
\newblock {\em Transactions of the American Mathematical Society}, 352, 02
  2000.

\bibitem{shieldsbook}
P.~C. Shields.
\newblock {\em The ergodic theory of discrete sample paths}, volume~13 of {\em
  Graduate Studies in Mathematics}.
\newblock American Mathematical Society, Providence, RI, 1996.

\bibitem{1975Walters}
P.~Walters.
\newblock Ruelle's operator theorem and {$g$}-measures.
\newblock {\em Trans. Amer. Math. Soc.}, 214:375--387, 1975.

\bibitem{1978Walters}
P.~Walters.
\newblock Invariant measures and equilibrium states for some mappings which
  expand distances.
\newblock {\em Trans. Amer. Math. Soc.}, 236:121--153, 1978.

\end{thebibliography}

\end{document}